\documentclass[12pt]{amsart}
\usepackage{amssymb}
\usepackage{amsmath}
\usepackage{latexsym}
\usepackage{color}
\usepackage{enumerate}
\usepackage[all]{xy}

\newtheorem{De}{Definition}[section]

\newtheorem{Pro}[De]{Proposition}
\newtheorem{Le}[De]{Lemma}
\newtheorem{Co}[De]{Corollary}

\newcommand{\xto}[1]{\xrightarrow[]{#1}}
\newcommand\id{{\sf Id}}
\newcommand\bu{{\bullet}}
\newcommand\Cat{\mathsf{Cat}}
\newcommand\sC{\mathsf{C}}
\newcommand\sA{\mathsf{A}}
\newcommand\sK{\mathsf{K}}
\newcommand\bC{\bf{C}}
\newcommand\bGr{\bf{Gr}}
\newcommand\bMon{\bf{Mon}}
\newcommand\bA{\bf{A}}
\newcommand\Xmod{\bf{Xmod}}
\newcommand\Xbsmod{\bf{Xbsmod}}
\newcommand\Xsmod{\bf{Xsmod}}
\newcommand\Qu{\mathsf{Qu}}
\newcommand\Xbsmodgr{\Xbsmod_{\bGr}}

\renewcommand\1{^{-1}}

\begin{document}

\title[Internal categories and coressed objects in $\bMon$]
{On internal categories and crossed objects in the category of monoids}

\author{Ilia Pirashvili}

\maketitle

\begin{abstract}
  It is a well-known fact that the category $\Cat(\bC)$ of internal categories
  in a category $\bC$ has a description in terms of crossed modules, when
  $\bC=\bGr$ is the category of groups. The proof of this result heavily uses
  the fact that any split epimorphism decomposes as a semi-direct product. An
  equivalent statement does not hold in the category $\bMon$ of monoids. In a
  previous work on quadratic algebras, \cite{p5} I constructed an internal
  category in the category of monoids, see Section \ref{6quad}. Based on this
  construction, this paper will introduce the notion of a crossed semi-bimodule
  and show that it gives rise to an object in $\Cat(\bMon)$. I will also relate
  this new notion to the crossed semi-modules introduced earlier by A.
  Patchkoria \cite{APach}.
\end{abstract}

\section{Introduction}

The notion of a crossed module was introduced by Whitehead in his study of
homotopy theory \cite{w}. A crossed module is a group homomorphism
$\partial:K\to A$, together with an action of $A$ on the group $K$, satisfying
some conditions. Starting with such data, one can construct the category
$\Cat(\partial:K\to A)$. Objects are elements of $A$, while morphisms from $b$
to $a$ are essentially elements $x$ of $K$, such that $b=\partial(x)a$ and the
composition is induced by the multiplication in $K$.

One can readily checks that $\Cat(\partial:K\to A)$ is an internal category in
the category of groups, meaning that all structural maps are group
homomorphisms. It is well-known that in this way, one obtain an isomorphism of
the category of internal categories in the category of groups, and the category
of crossed modules \cite{maclane}.

This paper will introduce the notion of a crossed semi-bimodule, which is pair
of monoids $A,K$, and triple of maps $A\times K\xto{\circ}A$, $A\times
K\xto{\lambda} K$, $K\times A\xto{\rho} K$, satisfying certain conditions. We
will show that one can construct a category $\Cat(A,K,\circ)$ from such a data,
which is internal in the category of monoids. 

We will show that crossed semi-bimodules for which $\lambda$ is a trivial
homomorphism, is equivalent to the category of crossed semi-modules introduced
by A. Patchkoria \cite{APach}. We will also show that in the cases when both
monoids $A$ and $K$ are groups, our notion is essentially equivalent to a
crossed module. 

In the last section, we will describe a particular crossed module, which arose
during my work on quadratic étale algebras \cite{p5}. It is an example
which does not comes from crossed semi-modules.

\section{Crossed (semi-)modules and crossed semi-bimodules}

\subsection{Internal categories}

Let $\bA$ be a category with finite limits. Recall that an internal category
$\sC$ in $\bA$ is the following data:
\begin{itemize}
  \item [i)] A pair of objects $C_0$ and $C_1$ of $\bA$, called, respectively,
    the \emph{object of objects} and \emph{object of morphisms} of $\sC$.
  \item [ii)] Four morphisms $d^1_0, d^1_1:C_1\to C_0$, $s^0_0:C_0\to C_1$ and
    $d^2_1:C_2\to C_1$, called, respectively, the \emph{target}, \emph{source},
    \emph{unit} and \emph{composite}. Here, $C_2$ is the following pullback:
    \[\xymatrix{
      C_2\ar[r]^{d^2_2}\ar[d]_{d^2_0} &C_1\ar[d]^{d^1_0}\\C_1\ar[r]_{d^1_1}&C_0.
    }\]
    The morphisms $d_0^1,d_1^1,d_1^2$ and $s_0^0$ satisfy the standard
    simplicial identities. 
 \end{itemize}
 
\subsection{Crossed modules and internal categories in $\bGr$} \label{cm}

Let $\bGr$ be the category of groups and group homomorphisms. It was known since
the 60's that the internal categories in $\bGr$ are equivalent to crossed
modules. Crossed modules were introduced by Whitehead \cite{w} in the study of
homotopy theory. 

\begin{De}\label{121904}
  A \emph{crossed module} is a group homomorphism $\partial:K\to A$, together
  with a right action of the group $A$ on the group $K$, denoted by $x^a$, where
  $x\in K$, $a\in A$. It must satisfy the following equalities
  \begin{enumerate}
    \item $\partial(x^a)=a^{-1}\partial(x)a$,
    \item $y^{\partial x}=x^{-1} yx$
  \end{enumerate}
  for all $x,y\in K$ and $a\in A$.
\end{De}

Denote by $\Xmod$ the category of crossed modules. From this, we can construct
an internal category in $\bGr$ in the following way:

Let $(K,A,\partial)$ be a crossed module. Define $C_0=A$ and $C_1$ to be
$A\times K$ as a set. The group structure is defined by
\[(a,x)(b,y)=(ab,x^b\, y).\]
This construction is compare with \cite{maclane}.

The group $C_2$ is $A\times K\times K$ as a set with the following operation
\[(a,x,y)(b,u,v)=(ab, x^bu,y^{b\partial u}v).\]
Next, we put
\[s^0_0(a)=(a,1),\ d_i^1(a,x)=
  \begin{cases} a, & {\rm if} \quad i=0\\
    a\partial(x), & {\rm if} \quad i=1
  \end{cases}\]
\[s^1_i(a,x)=
  \begin{cases} (a,x,1) & {\rm if} \quad i=0\\
    (a,1,x) & {\rm if}\quad i=0\end{cases}\]
\[d^2_i(a,x,y)=
  \begin{cases} (a,x) & {\rm if} \quad i=0\\
    (a,xy) & {\rm if} \quad i=1\\
  (a\partial (x),y) & {\rm if} \quad i=0\end{cases}\]
\[s^1_i(a,x)=
  \begin{cases} (a,x,1) & {\rm if} \quad i=0\\
    (a,1,x) & {\rm if}\quad i=1.\end{cases}\]
One checks that all maps are group homomorphisms and in this way, obtains an
internal category in $\sf Gr$. This is denoted by $\Cat(K\xto{\partial} A)$. It
is well-known that any internal category in $\sf Gr$ is of this form, and the
functor
\[{\Xmod}\to\Cat(\bGr), \quad (K,A,\partial)\mapsto\Cat(K\xto{\partial} A)\]
is an equivalence of categories \cite{maclane}.

\subsection{Crossed semi-modules}

A. Patchkoria \cite{APach} introduced the notion of a crossed semi-module and
proved that the category of crossed semi-modules is equivalent to the category
of Schreier categorical objects in the category of monoids. 

\begin{De}
  A crossed semi-module \cite{APach} is a monoid homomorphism $\partial:K\to A$,
  together with a right action of the monoid $A$ on the monoid $K$, denoted by
  $x^a$, $x\in K$, $a\in A$. It must satisfy the following equalities
	\begin{itemize}
		\item [i)] $a\partial(x^a)=\partial(x)a$,
		\item [ii)] $yx^{\partial y}=xy$
	\end{itemize}
	for all $a\in A$, $x,y\in K$.
\end{De}

By defining morphisms in the analogous way to groups, we can form a category,
which we will denote by $\Xsmod$.

Starting from a crossed semi-module, one can construct an internal category in
$\bMon$, the category of monoids, in the same exact way as the crossed modules.

\section{Crossed semi-bimodules and internal categories in the category of
monoids}

\subsection{Definition and first consequences}

The following definitions are probably more appropriate for monoids than crossed
semi-modules.

\begin{De}\label{mu}
  A crossed semi-bimodule consists of monoids $A$ and $K$, together with maps
  \[A\times K\xto{\circ} A;\quad (a,x)\mapsto a\circ x,\]
  \[A\times K\xto{\lambda} K;\quad (a,x)\mapsto\ ^ax,\]
  \[K\times A\xto{\rho} K;\quad (x,a)\mapsto x^a.\]
  The first map makes $A$ into a right $K$-set, the second map makes $K$ into a
  left $A$-monoid and the third map makes $K$ into a right $A$-monoid.
  Moreoever, the following conditions need to hold for every $x,y\in K$ and
  $a,b\in A$:
	\begin{enumerate}
		\item $(^ax)^b=\ ^a(x^b)$,
		\item $(ab)\circ(^ax)=a(b\circ x)$,
		\item $(ab)\circ(x^b)=(a\circ x)b$,
		\item $(^ay)x^{b\circ y}=x^b(^{a\circ x}y).$
	\end{enumerate}
\end{De}

For simplicity, instead of saying that$(K,A,\circ,\lambda,\rho)$ is a crossed
semi-bimodule, we will say that $(K,A,\circ)$ is a crossed semi-bimodule

\begin{De}
  A morphism from a crossed semi-bimodule $(K,A,\circ)$ to a crossed
  semi-bimodule $(K',A',\circ)$ is given by the homomorphisms $\kappa:K\to K'$
  and $\alpha:A\to A'$, for which the following conditions hold
  \begin{enumerate}
    \item $\alpha(a\circ x)=\alpha(a)\circ\kappa(x)$,
    \item $\kappa(^ax)=\ ^{\alpha(a)}\kappa(x)$,
    \item $\kappa(x^a)=\kappa(x)^{\alpha(a)}$.
  \end{enumerate}
\end{De}

Denote the category of crossed semi-bimodules by $\Xbsmod$. We will need also a
weaker version.

\begin{De}
  A weak morphism from a crossed semi-bimodule $(K,A,\circ)$ to a crossed
  semi-bimodule $(K',A',\circ)$ is given by a monoid homomorphism $\kappa:A\to
  A'$ and by the map $\gamma:A\times K\to K'$, such that the following
  conditions hold
  \begin{enumerate}
    \item $\gamma(a,xy)=\gamma(a,x)\gamma(a,y)$,
    \item $\kappa(a)\circ\gamma(a,x)=\kappa(a\circ x)$,
    \item $^{\kappa(a)}\gamma(b,y)\gamma(a,x)^{\kappa(a)\circ\gamma(b,y)}
      = \gamma(ab,\, ^ayx^{b\circ y})$.
	\end{enumerate}
\end{De}

Any morphism $(\kappa,\alpha)$ gives rise to a weak morphism $(\kappa,\gamma)$,
where $\gamma(a,x)=\alpha(x)$. Weak morphisms can be composed as follows
\[(\kappa',\gamma')\circ(\kappa,\gamma)=(\kappa'\circ\kappa,\gamma''),\]
where $\gamma''(a,x)=\gamma'(\kappa(a),\gamma(a,x))$. Thus, we obtain the
category $\Xbsmod_w$. Isomorphisms in $\Xbsmod_w$ are called weak isomorphisms.

Recall that a \emph{semi-bimodule} over $A$ is a commutative monoid $K$,
together with left and right monoid actions of $A$ on $K$, satisfying Equality
i) of Definition \ref{mu}.

\begin{Le}\label{07052023}
  The category of semi-bimodule over a monoid $A$ is equivalent to the full
  subcategory of crossed semi-bimodules over $A$, for which
  \[a\circ x=a\]
  holds for all $a\in A$ and all $x\in K$.
\end{Le}

The proof is just a straightforward verification.

\begin{Le}\label{1515042023}
  Let $(K,A,\circ)$ be a crossed semi-bimodule. Define
  \[\partial:K\to A\]
  by $\partial (y)=1\circ y$. For any $x,y\in K$, one has
  \[yx^{\partial(y)}=x( ^{\partial(x)}y).\]
\end{Le}

\begin{proof}
  Substitute $a=1=b$ in Identity iv) of the definition of a crossed
  semi-bimodule \ref{mu} to obtain $yx^{1\circ y}=\, x(^{1\circ x}y)$, and the
  result follows.
\end{proof}

\subsection{Crossed semi-bimodules and internal categories in $\bMon$}

\begin{Pro}\label{21020402023}
  Let $(K,A,\circ)$ be a crossed semi-bimodule.
  \begin{itemize}
    \item[i)] Define the multiplication $\bu$ on $A\times K$ by
      \[(a,x)\bu(b,y)=(ab,\,^ayx^{b\circ y}).\]
      This defines a monoid structure on the set $A\times K$. This monoid is
      denoted by $A\bowtie K$. 
    \item[ii)] Consider the set $A\times K\times K$ with the multiplication
			\[(a,x,y)\bu(b,u,v)
      =(ab,\,^aux^{b\circ u},\,^{a\circ u}vy^{b\circ u\circ v}).\]
      This induces a monoid structure denoted by $A\bowtie K\bowtie K$.
    \item[iii)] Define the maps
      \[\xymatrix{A\bowtie K\bowtie K\ar@/^/[r]^{d_0^2}\ar@/_/[r]_{d_0^2}\ar[r]
      & A\bowtie K\ar@/^/[r]^{d_0^1}\ar@/_/[r]_{d_0^1} & A; & A \ar[r]^{s^0_0}
      & A\bowtie K\ar@/^/[r]^{s_0^1}\ar@/_/[r]_{s_1^1} & A\bowtie K\bowtie K}\]
			by
      \[s^0_0(a)=(a,1),\quad s^1_0(a,x)=(a,x,1),\quad s^1_1(a,x)=(a,1,x).\]
      \[d^1_0(a,x)=a,\quad d^1_1(a,x)=a\circ x,\]
      \[d^2_0(a,x,y)=(a,x),\quad d^2_1(a,x,y)=(a,xy),\quad d^2_2(a,x,y)=(a\circ
      x, y).\]
			These maps satisfy all the simplicial identities and are homomorphisms.
    \item[iv)] The diagram 
			\[\xymatrix{A\bowtie K\bowtie K\ar[r]^{d^2_2}\ar[d]_{d^2_0} &
      A\bowtie K\ar[d]^{d^1_0} \\ A\bowtie K \ar[r]_{d^1_1} & A}\]
			is a pullback diagram.
    \item[v)] The above defines an internal category, denoted
      $\Cat(A,K,\circ)$, in the category $\bMon$ of monoids.
  \end{itemize}
\end{Pro}

\begin{proof}
  i) It is clear that $(1,1)$ serves as the unit of $\bu$. For associativity,
  we have
  \begin{eqnarray*}
    \left((a,x)(b,y)\right)(c,z)&=&(ab,\,^ayx^{b\circ y})(c,z)\\
                                &=&(abc,\,^{ab}z(\,^ayx^{b\circ y})^{c\circ z})\\
                                &=&(abc,\,^{ab}z(^ay)^{c\circ z}x^{(b\circ y)
                                   (c\circ z)}).
  \end{eqnarray*}
	On the other hand
  \begin{eqnarray*}
    (a,x)\left((b,y)(c,z)\right)&=&(a,x)(bc,\,^bzy^{c\circ z})\\
                                &=&(abc,\,^{ab}z(^ay)^{c\circ z}x^{(bc)\circ
                                   (^bzy^{c\circ z})}.
  \end{eqnarray*}
  Thus, we need to check the identity
  \begin{equation}\label{igive02042023}
    (b\circ y)(c\circ z)=(bc)\circ(^bzy^{c\circ z}).
  \end{equation}
  Equalities (2) and (3) of Definition \ref{mu} $a=b\circ y, b=c,x=z$ yield
  \begin{eqnarray*}
    (b\circ y)(c\circ z)&=&\left ((b\circ y)c\right)\circ(^{b\circ y}z)\\
                        &=&(bc)\circ y^c \circ (^{b\circ y}z).
  \end{eqnarray*}
  Equality (4) of the same Definition now gives Identity (\ref{igive02042023}).

  ii) This follow from parts iii) and iv).

  iii) We will only check that the maps $d^1_1,d^2_1, d^2_2$ are homomorphisms.
  The rest is clear.

  \noindent
  {\bf First claim}: $d^1_1$ is a homomorphism. Thanks to Identity
  (\ref{igive02042023}), we have
  \begin{eqnarray*}
    d_1^1(a,x)d^1_1(b,y)&=&(a\circ x)(b\circ y)\\
                        &=&(ab)\circ (^ayx^{b\circ y})\\
                        &=&d_1^1\left((a,x)\bu(b,y)\right).
  \end{eqnarray*}

  \noindent
  {\bf Second claim}: $d^2_1$ is a homomorphism. We have
  \begin{eqnarray*}
    d^2_1\left((a,x,y)\bu(b,u,v)\right)&=&d^2_1(ab,\,^aux^{b\circ u},\,
    ^{a\circ u}vy^{b\circ u\circ v}) \\
                                       &=&(ab, \, ^aux^{b\circ u}\,
                                       ^{a\circ u}vy^{b\circ u\circ v}).
  \end{eqnarray*}
  On the other hand
  \begin{eqnarray*}
    d^2_1(a,x,y)\bu d^2_1(b,u,v)&=&(a,xy)\bu(b,uv)\\
                                &=&(ab,\, ^a(uv)(xy)^{b\circ uv})\\
                                &=&(ab,\, ^au\, ^av \, x^{b\circ uv} y^{b\circ
                                uv}).
	\end{eqnarray*}
  Since $\circ$ is an action, we have $b\circ(uv)=b\circ u\circ v$, and based
  on Equality (4) of the definition of a crossed semi-bimodule \ref{mu}, we
  conclude that these expressions are equal.

  \noindent
  {\bf Third claim}: $d^2_2$ is a homomorphism. As for the first claim, we have
  \begin{eqnarray*}
    d^2_2\left((a,x,y)\bu(b,u,v)\right)&=&d^2_2(ab,\,^aux^{b\circ u},\,
    ^{a\circ u}vy^{b\circ u\circ v})\\
                                       &=&((a\circ x)(b\circ u,^{a\circ u}
                                       vy^{b\circ u\circ v})\\
                                       &=&(a\circ x,y)\bu(b\circ u,v)\\
                                       &=&d^2_2(a,x,y)\bu d^2_2(b,u,v).
	\end{eqnarray*}

  iv) Take two elements $(a,x)$ and $(b,y)$ in $A\bowtie K$ such that
  $d^1_1(a,x)=d^1_0(b,y)$. Equivalently $a\circ x=b$. Thus, there exits a
  unique element in $A\bowtie K\bowtie K$, namely $(a,x,y)$, such that
  $d^2_0(a,x,y)=(a,x)$ and $d^2_2(a,x,y)=(b,y)$. The result follows.

  v) This is a consequence of iii) and iv).
  \end{proof}

It is clear that if $(\kappa,\alpha)$ is a morphism of crossed semi-bimodules,
it defines an internal functor $\Cat(K,A,\circ)\to\Cat(K',A',\circ)$. In this
way, we obtain a functor $\Xbsmod\to\Cat(\bMon)$, from the category of crossed
semi-bimodules to the category of internal categories in the category of
monoids.

It is less obvious, but readily checked, that any weak morphism $(\kappa,
\gamma)$ from $(K,A,\circ)\to(K',A',\circ)$ also defines an internal functor
$\Cat(K,A,\circ)\to\Cat(K',A',\circ)$. On object objects, it is given by
$\kappa$. On morphism objects, it is defined by $(a,x)\mapsto(\kappa(a),
\gamma(a,x))$. On the object of composable morphisms, the internal functor is
given by $(a,x,y)\mapsto(\kappa(a),\gamma(a,x),\gamma(a,y))$. This gives us a
functor
\[\Xbsmod_w\to\Cat(\bMon).\]
	 
\begin{Le}\label{221504}
  Let $(K,A,\circ)$ be a crossed semi-bimodule. Define a new operation
  $\diamond$ on $K$ by
  \[x\diamond y:=yx^{\partial(y)}.\]
  Let $K^{tw}$ be the set $K$, equipped with the operation $\diamond$. Then
  $K^{tw}$ is a monoid.
\end{Le}

\begin{proof}
  By comparing $\diamond$ to the multiplication on $A\bowtie K$, we see that
  \[(1,x)\bu(1,y)=(1,x\diamond y).\]
  It follows that $\diamond$ is associative and $1\diamond x=x=x\diamond 1$. 
\end{proof}

\section{Relationship between crossed semi-modules and crossed semi-bimodules}

The aim of this section is to show that crossed semi-modules are exactly the
crossed semi-bimodules for which $^ax=x$, for every $a\in A$ and every $x\in K$.

\begin{Pro}
  \begin{itemize}
    \item[i)] Let $\partial:K\to A$ be a crossed semi-module. For elements $a\in
      M$ and $x\in G$, define
      \[^ax=x,\quad a\circ x=a\partial(x).\]
      This yields obtains a crossed semi-bimodule structure, denoted by
      $\Phi(K,A,\partial)$. 
    \item[ii)] Conversely, assume we have a crossed semi-bimodule for which 
      $^ax=a$, for all $a\in A$ and $x\in K$. Define $\partial:K\to A$ by
      \[\partial (x)=1\circ x.\]
      We obtain a crossed semi-module.
    \item[iii)] The constriction in i) defines a full and faithful functor
      $\Phi:\Xsmod\to\Xbsmod$ from the category of crossed semi-modules to the
      category of crossed semi-bimodules. Moreover, it induces an isomorphism 
      of $\Xsmod$ to the full subcategory of crossed semi-bimodules for which
      $^ax=x$, for all $a\in A$ and $x\in K$.
  \end{itemize}
\end{Pro}

\begin{proof}
  i) Equality (1) of Definition \ref{mu} is obvious. We have
  \[(ab)\circ(^a x)=(ab)\circ(x)=ab\partial(x)=a(b\circ x),\]
  which proves Equality (2) of Definition \ref{mu}. We also have
  \[(ab)\circ(x^b)=ab\partial(x^b)=a\partial(x)b=(a\circ x)b,\]
  which gives us Equality (3). Finally, we have 
  \[(^a y)x^{b\circ y}=yx^{b\partial(y)}=x^by=x^b(^{a\circ x}y),\]
	and part i) follows.
	
  ii) Assume $^ax=x$. By putting $b=1$ in Identity (2) of the definition of a
  crossed semi-bimodule Definition \ref{mu}, we obtain
  \[a\circ x=a\partial (x).\]
  By the definition of $\partial$ and setting $a=\partial(x)$, we get
	\[\partial(xy)=1\circ(xy)=(1\circ x)\circ y=\partial(x)\circ
  y=\partial(x)\partial (y).\]
  Hence, $\partial $ is a monoid homomorphism. It remains to show that the
  defining relations of crossed semi-modules are satisfied. To this end, we put
  $m=1$ in Identity (3) of Definition \ref{mu} to obtain
  \[b\partial(x^b) =b\circ x^b =(1\circ x)b=\partial(x)b.\]
  Quite similarly, taking $b=1$ in Identity (4) of Definition \ref{mu} gives us
  \[yx^{\partial(y)}=xy,\]
	and part ii) proved.
	
  iii) By our construction, in any crossed semi-bimodule coming from $\Xsmod$,
  we have $\partial(x)=1\circ x$. This implies that $\Phi$ is a full and
  faithful functor. The rest follows from part ii).
\end{proof}

\section{The case when $K$ and $A$ are groups}

Recall that if $(K,A,\circ)$ is a crossed semi-bimodule, we have a map
\[\partial:K\to A\]
defined by $\partial(x)=1\circ x$. 

In this section, we will investigate crossed semi-bimodules $(K,A,\circ)$, such
that $K$ and $A$ are groups. Our first observation is that the action $A\times
K\xto{\circ} A$ can be reconstructed from the map $\partial$.

\begin{Le}\label{4215042023}
  Let $(K,A,\circ)$ be a crossed semi-bimodule such that $K$ and $A$ are
  groups. For all $x,z\in K$ and $a\in A$, one has
  \begin{itemize} 
    \item [i)]$ a\circ x=a\partial(^{a\1}x)$,
    \item [ii)] $\partial(\,^{a\1}z^a)=a\1\partial(z)a$,
    \item [iii)] $\partial(xz^{\partial(x)})=\partial(z)\partial(x)$.
  \end{itemize} 
\end{Le}

\begin{proof}
  i) Putting $x=\,^{a\1}t$ and $b=1$ in Equality ii) of Definition \ref{mu}, we
  obtain $a\circ t=a\partial(^{a\1}t) $.

  ii) Having in mind i), Identity 3 of Definition \ref{mu} can be rewritten as
	\[ab\partial(^{(ab)\1}x^b)=a\partial(^{a\1}x)b.\]
	If we cancel both sides by $a$ and replace $x$ by $^{a\1}y$, we obtain 
  \[b\partial(^{b\1}y^b)=\partial(y)b,\]
	and the result follows.
	
	iii) Since $\circ$ is an action of $K$ on $A$, we have
  \[a\circ (xy)=(a\circ x)\circ y.\]
	Take $a=1$ and use Identity i) to rewrite the last equality as
	\begin{equation}
    \partial(xy)=\partial(x)\partial(^{(\partial(x))\1}y).
  \end{equation}
	Take $y=z\,^{\partial(x)}$ and use part ii) to obtain
  \begin{eqnarray*}
  \partial(xz\,^{\partial(x)})&=&
  \partial(x)\partial(^{\partial(x)\1}z\,^{\partial(x)})\\
                              &=&\partial(x)(\partial(x))\1\partial(z)\partial(x)\\
                              &=&\partial(z)\partial(x).
  \end{eqnarray*}
\end{proof}

\begin{Pro}
  Let $K$ and $A$ be groups. Assume there are given maps
  \[A\times K\xto{\lambda}K;\quad(a,x)\mapsto\ ^ax,\]
  \[K\times A\xto{\rho}K;\quad(x,a)\mapsto x^a\]
  such that the first map makes $K$ into a left $A$-monoid, the second map makes
  $K$ into a right $A$-monoid and these actions are compatible. That is to say
  \[(^ax)^b=\ ^a(x^b),\quad a,b\in A,x\in K.\]
  Additionally, assume there is given a map $\partial:K\to A$, such that the
  following identities hold:
	\begin{itemize}
		\item[i)] $\partial(xy)=\partial(x)\partial(\,^{\partial(x)\1}y)$
		\item[ii)] $\partial(\,^{b\1}z^b)=b\1\partial(z)b$
		\item[iii)] $yx^{\partial(y)}=x\,^{\partial(x)}y$.
	\end{itemize}
  Define the map $A\times K\xto{\circ}A$ by $ a\circ x=a\partial(^{a\1}x)$. This
  gives us a crossed semi-bimodule. Moreover, any crossed semi-bimodule for
  which $K$ and $A$ are groups can be obtained in this manner.
\end{Pro}

\begin{proof}
  By putting $x=1=y$ in i), we see that $\partial(1)=1$. Next, we show that
  $\circ$ defines a right action of $K$ on $A$. Indeed, we have $a\circ 1=
  a\partial(^{a\1}1)=a$. We also have
  \[a\circ (xy)= a\partial(^{a\1}x\, ^{a\1}y).\]
  Putting $u=\,^{a\1}x$ and $v=\,^{a\1}y$, the previous equality takes on the
  form $a\circ (xy)= a\partial(uv)$. Using Identity i) yields
  \[a\circ(xy)=a\partial(u)\partial\left(^{\partial(u)\1}v\right).\]
  On the other hand
  \begin{eqnarray*}
    (a\circ x)\circ y&=&(a\partial( ^{a\1}x))\circ y\\
                     &=&(a\partial(u))\circ y\\
                     &=&a\partial(u)\partial\left( ^{\partial(u)\1a\1}y\right)\\
                     &=&a\partial(u)\partial\left( ^{\partial(u)\1}v\right).
  \end{eqnarray*}
  Comparing these expressions we see that
  \[a\circ (xy)=(a\circ x)\circ y.\]
  Thus, $\circ$ defines the action of $K$ on $A$. It remains to check the
  identities (ii)-(iv) of Definition \ref{mu}. We have
  \begin{eqnarray*}
    (ab)\circ(^ax)&=&ab\partial(^{b\1a\1a}x)\\
                  &=&ab\partial(^{b\1}x)\\
                  &=&a(b\circ x).
  \end{eqnarray*}
  This proves Identity (2) of the Definition of a corssed semi-bimodule
  \ref{mu}. We also have
  \begin{eqnarray*}
    (ab)\circ x^b&=&ab\partial(^{b\1a\1}x^b)\\
                 &=&abb\1\partial(^{a\1}x)b\\
                 &=&a\partial(^{a\1}x)b\\
                 &=&(a\circ x)b.
  \end{eqnarray*}
  and Identity (3) of Definition \ref{mu} follows. Here, we used Identity ii).
  To prove Identity (4) of Definition \ref{mu}, we first show it in the
  particular case when $a=1$. Setting $z=y^{b\1}$ gives us
  \begin{eqnarray*}
    yx^{b\circ y}&=&yx^{b\partial(^{b\1}y)}\\
                 &=&z^bx^{b\partial(^{b\1}z^{b})}\\
                 &=&z^bx^{\partial(z)b}\\
                 &=&(zx^{\partial(z)})^b.
  \end{eqnarray*}
  Using Identity iii) now gives us
  \begin{eqnarray*}
    yx^{b\circ y}&=&(x(^{\partial x}z))^b\\
                 &=&x^b\,^{\partial(x)}z^b\\
                 &=&x^b\,^{\partial(x)}y.
  \end{eqnarray*}
  For the general case, put $t=\,^{a\1}x$. Thus, $x=\,^at$ and we have
  \begin{eqnarray*}
    \,^ayx^{b\circ y}&=&\,^ay(^at^{b\circ y})\\
                     &=&\,^a(y t^{b\circ y}).
  \end{eqnarray*}
  The case $a=1$ now gives us
  \begin{eqnarray*}
    \,^ayx^{b\circ y}&=&\,^a(t^b\,^{\partial(t)}y)\\
                     &=&x^b\,^{a\partial(t)}y\\
                     &=&x^b\,^{a\partial(^{a\1}x)}y\\
                     &=&\, x^b\, (^{a\circ x }y).
  \end{eqnarray*}
  This proves Identity (3) of Definition \ref{mu}.
\end{proof}

\begin{Pro}\label{431904}
  Let $(K,A,\circ)$ be a crossed semi-bimodule such that $K$ and $A$ are groups.
  Recall that we defined a monoid $K^{st}$ in Lemma \ref{221504}, which, as a
  set, is $K$. The monoid structure $\diamond$, however, is given by
  \[x\diamond y:= yx^{\partial(y)}=x\,^{\partial(x)}y,\]
  where $\partial(y)=1\circ y$, as usual.
 
 i) The map $\partial:K^{tw}\to A$ is a homomorphism.
 
 ii) Define a new action of $A$ on the $K^{st}$ by
 \[x^{*a}=\, ^{a\1}x^a,\]
 This gives us 
 \[\partial(x^{*a})=a\1\partial(x)a.\]
 
 iii) The monoid $K^{tw}$ is, indeed, a group.
 
 iv) We have
 \[y\diamond x^{*\partial y}=x\diamond y.\]
\end{Pro}

\begin{proof}
  i) We have
  \begin{eqnarray*}
    \partial(x\diamond y)&=&\partial(yx^{\partial(y)})\\
                         &=&\partial(x)\partial(y),
  \end{eqnarray*}
	according to part iii) of Lemma \ref{4215042023}.
	
	ii) This is a restatement of part ii) of Lemma \ref{4215042023}.
	
	iii) Put
  \[x^\flat=\,^{\partial(x)\1}(x\1)\]
	We have
  \begin{eqnarray*}
    x\diamond x^\flat&=&x(^{\partial(x)}(x^\flat))\\
                     &=&x(x^{-1})\\
                     &=&1.
  \end{eqnarray*}
	Thus, the $\diamond$-inverse of $x$ is $x^\flat$.
	
	iv) We have
  \begin{eqnarray*}
  y\diamond x^{*\partial y}&=&y\diamond(^{{\partial y}\1}x^{\partial(y)})\\
                           &=&y\,^{\partial(y)}(^{{\partial y}\1}x^{\partial(y)})\\
                           &=&y x^{\partial(y)}.
  \end{eqnarray*}
\end{proof} 

\begin{Co}
  Let $(K,A,\circ)$ be a crossed semi-bimodule such that $K$ and $A$ are
  groups. Then $(K^{tw},A,\partial)$ is a crossed module.
\end{Co}

\begin{proof}
  We already proved that $K^{tw}$ is a group and $\partial$ is a homomorphism,
  see Proposition \ref{431904}. The group $A$ acts on $K$ via $x^{*a}=\,
  ^{a\1}x^a$. This action respects the $\diamond$ operation. Indeed, we have
  \begin{eqnarray*}
    (x\diamond y)^{*a}&=&(yx^{\partial(y)})^{*a}\\
                      &=&\,(^{a\1}y^a)\,(^{a\1}(x^{\partial y})^a).
  \end{eqnarray*}
  We also have
  \begin{eqnarray*}
    (x^{*a})\diamond(y^{*a})&=&(^{a\1}x^{a})\diamond(^{a\1}y^{a})\\
                            &=&(^{a\1}y^{a})(^{a\1}x^{a})^{\partial(^{a\1}y^{a})}\\
                            &=&(^{a\1}y^{a})(^{a\1}x^{\partial(y)a}).
  \end{eqnarray*}
  Comparing these expressions, we see that $(x\diamond y)^{*a}=(x^{*a})\diamond
  (y^{*a})$. Finally Identities i) and ii) of Definition \ref{121904} hold,
  thanks to parts ii) and iv) of Proposition \ref{431904}.
\end{proof}

We constructed a functor $\Xmod\to\Xbsmod$ in Lemma \ref{07052023}, whose image
lies in the subcategory $\Xbsmodgr$ whose objects are such crossed
semi-bimodules, for which both $K$ and $A$ are groups. Moreover, in Proposition
\ref{431904}, we constructed the functor $\Xbsmodgr\to\Xmod$ in the reverse
direction. It is easy to check that the composite
\[\Xmod\to\Xbsmodgr\to\Xmod\]
is isomorphic to the identity functor. The next Lemma will shows that, though
the composite
\[\Xbsmodgr\to\Xmod\to\Xbsmodgr\]
is in general not isomorphic to the identity, it is weakly isomorphic to it.

\begin{Pro}
  Let $(K,A,\circ)$ be a crossed semi-bimodule such that $K$ and $A$ are groups.
  Denote by $(K^{tw},A,\partial)$ the crossed semi-bimodule corresponding to the
  crossed module $(K^{tw},A,\circ)$, constructed in Proposition \ref{431904}.
  
  The maps $\id:A\to A$ and $\gamma:A\times K^{tw}\to K$ define a weak
  isomorphism $(K^{tw},A,\partial)\to(K,A,\partial)$, where $\gamma(a,x)=\,^ax$.
\end{Pro}

\begin{proof}
  The fact that these maps define a weak morphism is a straightforward
  computation. It is an isomorphism as the inverse morphism is given by
  $\id:A\to A$ and $\gamma^{[-1]}:A\times K\to K^{tw}$, where the latter is
  defined by $\gamma^{[-1]}(a,x)=\,{a^{-1}}x$.
\end{proof}

\section{A construction related to rank two commutative algebras}\label{6quad}

The notion of crossed semi-bimodules arose during my work on rank two
commutative algebras \cite{p5}. In more detail, let $R$ be a commutative ring
and $p,q\in R$, such that
\[pq+2=0.\]
We define a crossed semi-bimodule, denoted by $\Qu(R):=(\sK(R),\sA(R),\circ)$.
Here, $\sK(R)$ is the multiplicative monoid of $2\times 2$ matrices of the form
\[\begin{pmatrix}1&r\\ 0& s\end{pmatrix},\]
where $r,s\in R$. The monoid $\sA(R)$, as a set, is simply $R^2$. Elements of
$\sA(R)$ are denoted by $[a,b]$. The monoid operation is defined by
\[[a,b][c,d]=[ac,a^2d+bc^2+p^2bd].\]
One easily sees that in this way, we obtain a commutative monoid with $[1,0]$ as
the unit.

\begin{Le}
  For elements
  \[k=\begin{pmatrix}1 & r \\ 0 & s\end{pmatrix}
  \in\sK(R)\quad{\rm and}\quad m=[a,b]\in\sA(R),\]
	we set 
  \[m\circ k=[as-pr,s^2b-qrsa-r^2].\]
	This defines an action of the monoid $\sK(R)$ on the set $\sA(R)$.
\end{Le}

\begin{proof}
  Taking $r=0$ and $s=1$, we see that $m\circ 1=m$. Next, take $k_1=
  \begin{pmatrix}1&r_1\\ 0& s_1\end{pmatrix}$.
  We have
	\begin{eqnarray*}
  (m\circ k)\circ k_1&=&[as-pr,s^2b-qrsa-r^2]\circ
  \begin{pmatrix} 1 & r_1 \\ 0 & s_1\end{pmatrix}\\
                    &=&[(as-pr)s_1-pr_1,s_1^2(s^2b-qrsa-r^2)-qr_1s_1(as-pr)-r_1^2].
	\end{eqnarray*}
	On the other, hand we have
	\begin{eqnarray*}
    m\circ(kk_1)&=&[a,b]\circ\begin{pmatrix} 1&r_1 + rs_1\\ 0&ss_1\end{pmatrix}\\
                &=&[ass_1-p(r_1+rs_1),s^2s_1^2b-q(r_1+rs_1)ss_1a-(r_1+rs_1)^2].
	\end{eqnarray*}
  Comparing these expression and using the fact that $pq=-2$, we see that
  \[(m\circ k)\circ k_1=m\circ(kk_1),\]
	implying the lemma.
\end{proof}

\begin{Pro}
  For elements
  \[k=\begin{pmatrix} 1 & r \\ 0 & s\end{pmatrix}\in\sK(R)\quad {\rm and}\quad
  m=[a,b]\in\sA(R)\]
	we set 
  \[k^m=\begin{pmatrix} 1 & r a \\ 0 & s\end{pmatrix}=\ ^mk.\]
	In this way, we obtain a crossed semi-bimodule $(\sK(R),\sA(R),\circ)$.
\end{Pro}

\begin{proof}
  For $n=[c,d]$, we have
	\begin{eqnarray*}
    (mn)\circ(^mk)&=&[ac,a^2d+bc^2+p^2bd]\circ
    \begin{pmatrix} 1 & ra \\ 0 & s\end{pmatrix}\\
                  &=&[acs-pra,s^2(a^2d+bc^2+p^2bd)-qarsac-r^2a^2]
  \end{eqnarray*}
	On the other hand
	\begin{eqnarray*}
    m(n\circ k)&=&[a,b][cs-pr,s^2d-qrsc-r^2]\\
               &=&[acs-par,a^2(s^2d-qrsc-r^2)+b(cs-pr)^2+p^2b(s^2d-qrsc-r^2)].
  \end{eqnarray*}
	Since $pq=-2$, we obtain
  \[(mn)\circ (^ms)=m(n\circ s).\]
	For
  \[h=\begin{pmatrix} 1 & u \\ 0 & v \end{pmatrix}\in K(R),\]
	we have
	\begin{eqnarray*}
    (^mh)k^{n\circ h} &=&
    \begin{pmatrix}1&ua\\ 0& v\end{pmatrix}
    \begin{pmatrix}1&crv-pur\\ 0& s\end{pmatrix}\\
                    &=&
    \begin{pmatrix}1& crv-pur+asu\\ 0& sv\end{pmatrix}\\
                    &=&
    \begin{pmatrix}1&rc\\ 0 & s\end{pmatrix}
    \begin{pmatrix}1& asu-pru\\ 0&v\end{pmatrix}\\
                    &=& k^n(^{m\circ k}h)
	\end{eqnarray*}
\end{proof}

The internal category in $\bMon$, corresponding to $(\sK(R),\sA(R),\circ)$, is
intimately related to the Galois algebras over the Hopf $R$-algebra $J^{pq}$,
thanks to \cite{p5}. The latter is freely generated by $1$ and $x$ as an
$R$-module. The multiplication and comultiplication in $J^{pq}$ is given by
\[x^2=qx,\quad\Delta(x)=1\otimes x+x\otimes 1+p(x\otimes x).\]

\end{document}